\documentclass{amsart}
\usepackage{hyperref}

 \newtheorem{theorem}{Theorem}[section]
 \newtheorem{corollary}[theorem]{Corollary}
 \newtheorem{lemma}[theorem]{Lemma}
 
 \theoremstyle{definition}
 
 \theoremstyle{remark}
 \newtheorem{remark}[theorem]{Remark}
 
 \numberwithin{equation}{section}

\newcommand{\eps}{\varepsilon}

\newcommand{\C}{\mathbb{C}}
\newcommand{\N}{\mathbb{N}}
\newcommand{\R}{\mathbb{R}}
\newcommand{\T}{\mathbb{T}}
\newcommand{\Z}{\mathbb{Z}}

\newcommand{\cC}{\mathcal{C}}
\newcommand{\cG}{\mathcal{G}}
\newcommand{\cH}{\mathcal{H}}
\newcommand{\cK}{\mathcal{K}}
\newcommand{\cR}{\mathcal{R}}
\newcommand{\diag}{\operatorname{diag}}
\newcommand{\im}{\operatorname{Im}}
\begin{document}
%
%
%
%
%
%
\title[Asymptotics of Toeplitz matrices]
{Asymptotics of Toeplitz Matrices with Symbols in Some Generalized Krein Algebras}

\author[A. Yu. Karlovich]{Alexei Yu. Karlovich}
\address{
Departamento de Matem\'atica\\
Faculdade de Ci\^encias e Tecnologia\\
Universidade Nova de Lisboa\\
Quinta da Torre\\
2829--516 Caparica\\
Portugal}
\email{oyk@fct.unl.pt}
\thanks{This work is partially supported by the grant FCT/FEDER/POCTI/MAT/59972/2004.}


\dedicatory{To the memory of Mark Krein (1907--1989)}
\begin{abstract}
Let $\alpha,\beta\in(0,1)$ and
\[
K^{\alpha,\beta}:=\left\{a\in L^\infty(\T):\
\sum_{k=1}^\infty |\widehat{a}(-k)|^2 k^{2\alpha}<\infty,\
\sum_{k=1}^\infty |\widehat{a}(k)|^2 k^{2\beta}<\infty
\right\}.
\]
Mark Krein proved in 1966 that $K^{1/2,1/2}$ forms a Banach algebra.
He also observed that this algebra is important in the asymptotic theory
of finite Toeplitz matrices. Ten years later,  Harold Widom extended
earlier results of Gabor Szeg\H{o} for scalar symbols and
established the asymptotic trace formula
\[
\operatorname{trace}f(T_n(a))=(n+1)G_f(a)+E_f(a)+o(1)
\quad\text{as}\ n\to\infty
\]
for finite Toeplitz matrices $T_n(a)$ with matrix symbols $a\in K^{1/2,1/2}_{N\times N}$.
We show that if $\alpha+\beta\ge 1$ and $a\in K^{\alpha,\beta}_{N\times N}$,
then the Szeg\H{o}-Widom asymptotic trace formula holds with $o(1)$ replaced by
$o(n^{1-\alpha-\beta})$.
\end{abstract}
\maketitle
\section{Introduction and the main result}
For $1\le p\le \infty$, let $L^p:=L^p(\T)$ and $H^p:=H^p(\T)$ be the standard Lebesgue
and Hardy spaces on the unit circle $\T$, respectively. Denote by $\{\widehat{a}(k)\}_{k\in\Z}$
the sequence of the Fourier coefficients of a function $a\in L^1(\T)$,
\[
\widehat{a}(k)=\frac{1}{2\pi}\int_0^{2\pi}a(e^{i\theta})e^{-ik\theta}d\theta
\quad (k\in\Z).
\]
For $\alpha,\beta\in(0,1)$, put
\[
\begin{split}
K^{\alpha,0} &:=\left\{a\in L^\infty(\T):\
\sum_{k=1}^\infty |\widehat{a}(-k)|^2 k^{2\alpha}<\infty,
\right\},
\\
K^{0,\beta} &:=\left\{a\in L^\infty(\T):\
\sum_{k=1}^\infty |\widehat{a}(k)|^2 k^{2\beta}<\infty
\right\},
\\
K^{\alpha,\beta} &:=K^{\alpha,0}\cap K^{0,\beta}.
\end{split}
\]
It was Mark Krein \cite{Krein66} who first discovered that $K^{1/2,1/2}$
forms a Banach algebra under pointwise multiplication and the norm
\[
\|a\|_{1/2,1/2}:=\|a\|_{L^\infty}+
\left(\sum_{k=-\infty}^\infty |\widehat{a}(k)|^2(|k|+1)\right)^{1/2}.
\]
By the same method, one can show that if $\alpha,\beta\in[1/2,1)$, then
$K^{\alpha,0}$ and $K^{0,\beta}$ are Banach
algebras under pointwise multiplication and the norms
\[
\begin{split}
\|a\|_{\alpha,0} &:=\|a\|_{L^\infty}+\left(\sum_{k=0}^\infty
|\widehat{a}(-k)|^2(k+1)^{2\alpha}\right)^{1/2},
\\
\|a\|_{0,\beta} &:=\|a\|_{L^\infty}+\left(\sum_{k=0}^\infty
|\widehat{a}(k)|^2(k+1)^{2\beta}\right)^{1/2},
\end{split}
\]
respectively. Further, if $\max\{\alpha,\beta\}\ge 1/2$, then $K^{\alpha,\beta}$
is a Banach algebra under pointwise multiplication and the norm
\[
\|a\|_{\alpha,\beta} :=\|a\|_{L^\infty}
+\left(\sum_{k=0}^\infty|\widehat{a}(-k)|^2(k+1)^{2\alpha}\right)^{1/2}
+\left(\sum_{k=0}^\infty|\widehat{a}(k)|^2(k+1)^{2\beta}\right)^{1/2}
\]
(see \cite[Chap.~4]{BS83} and also \cite[Sections~10.9--10.11]{BS06} and
\cite[Theorem~1.3]{BKS07}).
In these sources even more general algebras are considered. The algebra $K^{1/2,1/2}$
is referred to as the \textit{Krein algebra}. The algebras $K^{\alpha,0}$, $K^{0,\beta}$,
and $K^{\alpha,\beta}$ will be called \textit{generalized Krein algebras}.

Suppose $N\in\N$. For a Banach space $X$, let $X_N$ and $X_{N\times N}$
be the spaces of vectors and matrices with entries in $X$, respectively.
Let $I$ be the identity operator, $P$ be the Riesz projection of
$L^2$ onto $H^2$, $Q:=I-P$, and define $I,P$, and $Q$ on
$L_N^2$ elementwise. For $a\in L_{N\times N}^\infty$ and $t\in\T$,
put $\widetilde{a}(t):=a(1/t)$ and $(Ja)(t):=t^{-1}\widetilde{a}(t)$.
Define \textit{Toeplitz operators}
\[
T(a):=PaP|\im P,
\quad
T(\widetilde{a}):=JQaQJ|\im P
\]
and \textit{Hankel operators}
\[
H(a):=PaQJ|\im P,
\quad
H(\widetilde{a}):=JQaP|\im P.
\]
The function $a$ is called the \textit{symbol} of $T(a)$, $T(\widetilde{a})$,
$H(a)$, $H(\widetilde{a})$. We are interested in the asymptotic behavior of
\textit{finite block Toeplitz matrices}
\[
T_n(a):=\big(\widehat{a}(j-k)\big)_{j,k=0}^n
\]
generated by (the Fourier coefficients of) the symbol $a$ as $n\to\infty$.
It should be noted that asymptotics of Toeplitz matrices was one of the topics
of Mark Krein's interests. In particular, he proved \cite{Krein66} that
$K^{1/2,1/2}$ is an optimal smoothness class for the validity of the
strong Szeg\H{o} limit theorem for scalar positive symbols (independently
this result was obtained by Devinatz \cite{Devinatz67}; for an extension
of this result to matrix positive definite symbols, see B\"ottcher and
Silbermann \cite{BS80}).
Many results about asymptotic properties of $T_n(a)$ as $n\to\infty$ are contained
in the books by Grenander and Szeg\H{o} \cite{GS58}, B\"ottcher and Silbermann
\cite{BS83,BS99,BS06}, Hagen, Roch, and Silbermann \cite{HRS01}, Simon \cite{Simon05},
and B\"ottcher and Grudsky \cite{BG05}.

Let $\operatorname{sp}A$ denote the spectrum of an operator $A$. If $f$ is an
analytic function in an open neighborhood of $\operatorname{sp}A$, then we will
simply say that $f$ is analytic on $\operatorname{sp}A$. We assume that the reader
is familiar with basics of trace class operators and their operator determinants
(see Gohberg and Krein \cite[Chap.~3 and 4]{GK69} or Section~\ref{sec:Borodin-Okounkov}).
If $A$ is a trace class operator, then $\operatorname{trace} A$ denotes the \textit{trace} of $A$
and $\det (I-A)$ denotes the \textit{operator determinant} of $I-A$.

The following result was proved by Widom \cite[Theorem~6.2]{Widom76}
(see also \cite[Section~10.90]{BS06}).
It extends earlier results by Szeg\H{o} (see \cite{GS58}) and now
it is usually called the Szeg\H{o}-Widom asymptotic trace formula.
\begin{theorem}[Widom] \label{th:Widom}
Let $N\ge 1$. If $a$ belongs to $K_{N\times N}^{1/2,1/2}$ and $f$ is analytic on
$\mathrm{sp}\,T(a)\cup\mathrm{sp}\,T(\widetilde{a})$, then
\begin{equation}\label{eq:Widom}
\operatorname{trace} f(T_n(a))=(n+1)G_f(a)+E_f(a)+o(1)
\quad\text{as }\; n\to\infty,
\end{equation}
where
\begin{align*}
G_f(a)
&:=
\frac{1}{2\pi}\int_0^{2\pi}(\operatorname{trace} f(a))(e^{i\theta})d\theta,
\\
E_f(a)
&:=
\frac{1}{2\pi i}\int_{\partial\Omega}f(\lambda)
\frac{d}{d\lambda}\log\det T[a-\lambda]T[(a-\lambda)^{-1}]d\lambda,
\end{align*}
and $\Omega$ is any bounded open set containing
$\mathrm{sp}\,T(a)\cup\mathrm{sp}\,T(\widetilde{a})$
on the closure of which $f$ is analytic.
\end{theorem}

Our main result is the following refinement of Theorem~\ref{th:Widom}, which
gives a higher order asymptotic trace formula.
\begin{theorem}\label{th:main}
Let $N\ge 1$ and $\alpha,\beta\in(0,1)$. Suppose that $\alpha+\beta\ge 1$.
If $a\in K_{N\times N}^{\alpha,\beta}$ and $f$ is analytic on
$\mathrm{sp}\,T(a)\cup\mathrm{sp}\,T(\widetilde{a})$,
then \eqref{eq:Widom} is true with $o(1)$ replaced by $o(n^{1-\alpha-\beta})$.
\end{theorem}
Notice that higher order asymptotic trace formulas are known for other classes
of symbols: see
\cite{VMS03} for $W\cap K^{\alpha,\alpha}$ with $\alpha>1/2$
(here $W$ stands for the Wiener algebra of functions with absolutely convergent Fourier series),
\cite{K-ZAA} for weighted Wiener algebras,
\cite{K-IWOTA05} for H\"older-Zygmund spaces,
\cite{K-WOAT06} for generalized H\"older spaces.
All these classes consist of continuous functions only. More precisely,
they are decomposing algebras of continuous functions in the sense of Budyanu
and Gohberg. An invertible matrix function in such an algebra admits a Wiener-Hopf
factorization within the algebra.
The proofs of \cite{K-ZAA,K-IWOTA05,K-WOAT06} are based on a combination of
this observation and an approach of B\"ottcher and Silbermann \cite{BS80}
(see also \cite[Sections~6.15--6.22]{BS83} and \cite[Sections~10.34--10.40]{BS06})
to higher order asymptotic formulas of Toeplitz determinants with Widom's original
proof of Theorem~\ref{th:Widom} (see \cite{Widom76} and \cite[Section~10.90]{BS06}).
As far as we know, Vasil'ev, Maximenko, and Simonenko have never published
a proof of the result stated in the short note \cite{VMS03}, however, their
result can be proved by the same method.

Generalized Krein algebras $K^{\alpha,\beta}$ may contain discontinuous functions.
To study them we need a more advanced factorization theory in decomposing
algebras of $L^\infty$ functions developed by Heinig and Silbermann \cite{HS84}.
We present main results of this theory in Section~\ref{sec:factorization} and
then apply them to $K^{\alpha,\beta}$ with $\alpha+\beta\ge 1$ and $\max\{\alpha,\beta\}>1/2$.
Under these assumptions, if both Toeplitz operators $T(a)$ and $T(\widetilde{a})$
are invertible, then $a$ admits simultaneously canonical right and left
Wiener-Hopf factorizations $a=u_-u_+=v_+v_-$ in $K_{N\times N}^{\alpha,\beta}$.
The factors and their inverses in these factorizations are stable under small
perturbations of $a$ in the norm of $K_{N\times N}^{\alpha,\beta}$. We will use
this fact in Section~\ref{sec:proof} for factorizations of $a-\lambda$, where
$\lambda$ belongs to a compact neighborhood $\Sigma$ of the boundary of a
set $\Omega$ containing $\operatorname{sp}T(a)\cup\operatorname{sp}T(\widetilde{a})$.

Section~\ref{sec:Borodin-Okounkov} contains some preliminaries on trace class
operators and their determinants. Further we formulate the Borodin-Okounkov
formula under weakened smoothness assumptions. This is an exact formula
which relates determinants of finite Toeplitz matrices $\det T_n(a)$ and
operator determinants of $I-Q_nH(b)H(\widetilde{c})Q_n$, where
$Q_nH(b)H(\widetilde{c})Q_n$ are truncations of the product of Hankel operators
$H(b)$ and $H(\widetilde{c})$ with $b:=v_-u_+^{-1}$ and $c:=u_-^{-1}v_+$.
Here $Q_n:=I-P_n$ and $P_n$ is the finite section projection.

If $a-\lambda\in K_{N\times N}^{\alpha,\beta}$, then we can effectively estimate
the speed of convergence of the trace class norm of $I-Q_nH[b(\lambda)]H[\widetilde{c(\lambda)}]Q_n$
to zero as $n\to\infty$ uniformly in $\lambda\in\Sigma$. This
speed is $o(n^{1-\alpha-\beta})$. Combining this estimate with the Borodin-Okounkov
formula for $a-\lambda$ and then applying Widom's ``differentiate-multiply-integrate"
arguments with respect to $\lambda\in\Sigma$, we prove Theorem~\ref{th:main}
in Section~\ref{sec:proof}.
\section{Wiener-Hopf factorization and generalized Krein algebras}
\label{sec:factorization}
\subsection{Wiener-Hopf factorization in decomposing algebras}
For a unital algebra $A$, let $\cG A$ denote the its group of invertible
elements.

Mark Krein \cite{Krein58} was the first to understand the Banach algebraic
background of Wiener-Hopf factorization and to present the method in a
crystal-clear manner. Gohberg and Krein \cite{GK58} proved that
$a\in \cG W_{N\times N}$ admits a Wiener-Hopf factorization. Later Budyanu and
Gohberg developed an abstract factorization theory in decomposing algebras
of \textit{continuous} functions. Their results are contained in
\cite[Chap.~2]{CG81}. Heinig and Silbermann \cite{HS84} extended the theory
of Budyanu and Gohberg to the case of decomposing algebras which may
contain \textit{discontinuous} functions. The following definitions and
results  are taken from \cite{HS84} (see also \cite[Chap.~5]{BS83}).

Let $A$ be a Banach algebra of complex-valued functions on the unit circle
$\T$  under a Banach algebra norm $\|\cdot\|_A$. The algebra $A$ is said to be
\textit{decomposing} if it possesses the following properties:

\begin{enumerate}
\item[(a)]
$A$ is continuously embedded in $L^\infty$;

\item[(b)]
$A$ contains all Laurent polynomials;

\item[(c)]
$PA\subset A$ and $QA\subset A$.
\end{enumerate}

Using the closed graph theorem it is easy to deduce from (a)--(c)
that $P$ and $Q$ are bounded on $A$ and that $PA$ and $QA$ are
closed subalgebras of $A$.
For $k\in\Z$ and $t\in\T$, put $\chi_k(t):=t^k$.
Given a decomposing algebra $A$ put
\[
A_+=PA,\quad
\stackrel{\circ}{A}_-=QA,\quad
\stackrel{\circ}{A}_+=\chi_1A_+,\quad
A_-=\chi_1\stackrel{\circ}{A}_-.
\]

Let $A$ be a decomposing algebra. A matrix function $a\in A_{N\times N}$
is said to \textit{admit a right} (resp. \textit{left})
\textit{Wiener-Hopf factorization in}
$A_{N\times N}$ if it can be represented in the form
\[
a=a_-da_+
\quad(\text{resp.}\ a=a_+da_-),
\]
where
$a_\pm\in \cG(A_\pm)_{N\times N}$ and
\[
d=\diag(\chi_{\kappa_1},\dots,\chi_{\kappa_N}),
\quad
\kappa_i\in\Z,
\quad
\kappa_1\le\kappa_2\le\dots\le\kappa_N.
\]
The integers $\kappa_i$ are usually called the \textit{right} (resp. \textit{left})
\textit{partial indices} of $a$; they can be shown to be uniquely determined by $a$.
If $\kappa_1=\dots=\kappa_N=0$, then the Wiener-Hopf factorization is
said to be \textit{canonical}.
A decomposing algebra $A$ is said to have the \textit{factorization
property} if every matrix function in $\cG A_{N\times N}$ admits a right
Wiener-Hopf factorization in $A_{N\times N}$.

Let $\cR$ be the restriction to the unit circle $\T$ of the set of all
rational functions defined on the whole plane $\C$ and having no poles
on $\T$.
\begin{theorem}\label{th:HS}
Let $A$ be a decomposing algebra. If at least one of the sets
\[
(\cR\,\cap\stackrel{\circ}{A_-})+A_+
\quad\text{or}\quad \stackrel{\circ}{A_-}+(\cR\cap A_+)
\]
is dense in $A$, then $A$ has the factorization property.
\end{theorem}
\subsection{Stability of factors and their inverses under small perturbations}
Let $A$ be a Banach algebra equipped with a norm $\|\cdot\|_A$. We will
always consider an admissible norm $\|\cdot\|_{A_{N\times N}}$ in $A_{N\times N}$.
Recall that a Banach algebra norm is said to be admissible (see \cite[Section 1.29]{BS06})
if there exist positive constants $m$ and $M$ such that
\[
m\max_{1\le i,j\le N}\|a_{ij}\|_A
\le
\|a\|_{A_{N\times N}}
\le
M\max_{1\le i,j\le N}\|a_{ij}\|_A
\]
for every matrix $a=(a_{ij})_{i,j=1}^N\in A_{N\times N}$.

The following result can be extracted from a stability theorem for factors
and their inverses in the Wiener-Hopf factorization in decomposing algebras
given in \cite[Theorem~6.15]{LS87}. There it was assumed, in addition, that
a decomposing algebra is continuously embedded in the set of all
continuous functions. However, the result is also true for decomposing
algebras in the sense of Heinig and Silbermann adopted in this paper.
\begin{theorem}\label{th:stability}
Let $A$ be a decomposing algebra and $N\ge 1$. Suppose $a,c\in A_{N\times N}$
both admit canonical right {\rm(}resp. left{\rm)} Wiener-Hopf factorizations in
$A_{N\times N}$. Then for any $\eps>0$ there exists a $\delta>0$ such that
if
\[
\|a-c\|_{A_{N\times N}}<\delta,
\]
then for every canonical right {\rm(}resp. left{\rm)} Wiener-Hopf factorization
$a=a_-^{(r)}a_+^{(r)}$ {\rm(}resp. $a=a_+^{(l)}a_-^{(l)}${\rm)}
one can choose a canonical right {\rm(}resp. left{\rm)} Wiener-Hopf factorization
$c=c_-^{(r)}c_+^{(r)}$ {\rm(}resp. $c=c_+^{(l)}c_-^{(l)}${\rm)} such that
\[
\|a_\pm^{(r)}-c_\pm^{(r)}\|_{A_{N\times N}}<\eps,
\quad
\|[a_\pm^{(r)}]^{-1}-[c_\pm^{(r)}]^{-1}\|_{A_{N\times N}}<\eps
\]
\[
(\text{resp.}\quad
\|a_\pm^{(l)}-c_\pm^{(l)}\|_{A_{N\times N}}<\eps,
\quad
\|[a_\pm^{(l)}]^{-1}-[c_\pm^{(l)}]^{-1}\|_{A_{N\times N}}<\eps\ ).
\]
\end{theorem}
\subsection{Invertibility in generalized Krein algebras}
For $1\le p\le\infty$, let $\overline{H^p}:=\{a\in L^p: \overline{a}\in H^p\}$ and let
$C:=C(\T)$ denote the set of all continuous functions on $\T$. If $\alpha,\beta\ge 1/2$, then
in view of \cite[Lemma~6.2]{BKS07},
\begin{equation}\label{eq:Krein-embedding}
K^{\alpha,0}\subset C+H^\infty,\quad K^{0,\beta}\subset C+\overline{H^\infty}.
\end{equation}
Hence, if $\alpha,\beta\in(0,1)$ and $\alpha+\beta\ge 1$, then
\begin{equation}\label{eq:Krein-embedding-2}
K^{\alpha,\beta}\subset (C+H^\infty)\cup (C+\overline{H^\infty}).
\end{equation}

The following result was proved by Krein \cite{Krein66} for $\alpha=\beta=1/2$.
\begin{theorem}[see {\cite[Theorem~1.4]{BKS07}}]\label{th:invertibility}
Let $\alpha,\beta\in(0,1)$.
\begin{enumerate}
\item[(a)]
Suppose $\alpha\ge 1/2$ and $K$ is either $K^{\alpha,0}$ or $K^{\alpha,1-\alpha}$.
If $a\in K$, then
\[
a\in \cG K\Longleftrightarrow a\in \cG(C+H^\infty).
\]

\item[(b)]
Suppose $\beta\ge 1/2$ and $K$ is either $K^{0,\beta}$ or $K^{1-\beta,\beta}$.
If $a\in K$, then
\[
a\in \cG K\Longleftrightarrow a\in \cG(C+\overline{H^\infty}).
\]
\end{enumerate}
\end{theorem}
\begin{corollary}\label{co:invertibility}
Let $\alpha,\beta\in(0,1)$.
\begin{enumerate}
\item[(a)]
Suppose $\alpha\ge 1-\beta\ge 1/2$.
If $a\in K^{\alpha,\beta}$, then
\[
a\in \cG K^{\alpha,\beta}\Longleftrightarrow a\in \cG(C+H^\infty).
\]

\item[(b)]
Suppose $\beta\ge 1-\alpha\ge 1/2$.
If $a\in K^{\alpha,\beta}$, then
\[
a\in \cG K^{\alpha,\beta}\Longleftrightarrow a\in \cG(C+\overline{H^\infty}).
\]

\item[(c)]
Suppose $\alpha\ge \beta\ge 1/2$ or $\beta\ge\alpha\ge 1/2$.
If $a\in K^{\alpha,\beta}$, then
\[
a\in \cG K^{\alpha,\beta}\Longleftrightarrow
a\in \cG\big((C+H^\infty)\cap(C+\overline{H^\infty})\big).
\]
\end{enumerate}
\end{corollary}
\begin{proof}
(a) Let $a\in K^{\alpha,\beta}=K^{1-\beta,\beta}\cap K^{\alpha,0}$.
By Theorem~\ref{th:invertibility}(a),
\[
a\in \cG K^{1-\beta,\beta}\Longleftrightarrow a\in \cG(C+H^\infty),
\quad
a\in \cG K^{\alpha,0}\Longleftrightarrow a\in \cG(C+H^\infty).
\]
Thus $a\in \cG K^{\alpha,\beta}\Longleftrightarrow a\in \cG(C+H^\infty)$.
Part (a) is proved. Part (b) follows form Theorem~\ref{th:invertibility}(b)
in the same fashion.

(c) Let $a\in K^{\alpha,\beta}=K^{\alpha,0}\cap K^{0,\beta}$. From
Theorem~\ref{th:invertibility} it follows that
\[
a\in \cG K^{\alpha,0}\Longleftrightarrow \cG(C+H^\infty),
\quad
a\in \cG K^{0,\beta}\Longleftrightarrow \cG(C+\overline{H^\infty}).
\]
Hence $a\in \cG K^{\alpha,\beta}=G(K^{\alpha,0}\cap K^{0,\beta})\Longleftrightarrow
a\in \cG\big((C+H^\infty)\cap(C+\overline{H^\infty})\big)$. Part (c) is proved.
\end{proof}
\subsection{Wiener-Hopf factorization in generalized Krein algebras}
\begin{lemma}\label{le:fact-property}
Let $\alpha,\beta\in(0,1)$ and $\max\{\alpha,\beta\}>1/2$. Then $K^{\alpha,\beta}$
is a decomposing algebra with the factorization property.
\end{lemma}
\begin{proof}
The statement is proved by analogy with \cite[Lemma~7.7]{BKS07}.
By \cite[Lemma~6.1]{BKS07}, the projections $P$ and $Q$
are bounded on $K^{\alpha,\beta}$. Hence $K^{\alpha,\beta}$ is a decomposing
algebra. Assume that $\beta>1/2$. Taking into account that
\[
K^{\alpha,\beta}=L^\infty\cap (QB_2^\alpha+PB_2^\beta),
\]
where $B_2^\alpha$ and $B_2^\beta$ are Besov spaces, from \cite[Sections~3.5.1 and 3.5.5]{ST87}
one can deduce that $\cR\cap PK^{\alpha,\beta}$ is dense in $PK^{\alpha,\beta}$.
Analogously, if $\alpha>1/2$, then $\cR\cap QK^{\alpha,\beta}$ is dense
in $QK^{\alpha,\beta}$. Theorem~\ref{th:HS} gives the factorization
property of $K^{\alpha,\beta}$.
\end{proof}
\begin{theorem}\label{th:Krein-WH}
Let $N\ge 1$, $\alpha,\beta\in(0,1)$, $\alpha+\beta\ge 1$, and $\max\{\alpha,\beta\}>1/2$.
If $a\in K_{N\times N}^{\alpha,\beta}$ and both $T(a)$ and $T(\widetilde{a})$
are invertible on $H_N^2$, then $a$ is invertible in $K_{N\times N}^{\alpha,\beta}$
and admits canonical right and left Wiener-Hopf factorizations in
$K_{N\times N}^{\alpha,\beta}$.
\end{theorem}
\begin{proof}
Once one has at hands Corollary~\ref{co:invertibility} and Lemma~\ref{le:fact-property},
the proof is developed as in \cite[Theorem~1.7(a)]{BKS07}. For the convenience
of the reader we give a complete proof here.

Suppose $\alpha=\max\{\alpha,\beta\}$. It is clear that for every $\beta\in(0,1)$
one has $\beta\ge 1/2$ or $1-\beta\ge 1/2$. Thus $\alpha\ge\beta\ge 1/2 \ (\ge 1-\beta)$
or $\alpha\ge 1-\beta\ge 1/2 \ (\ge\beta)$. In the first case from
\eqref{eq:Krein-embedding} it follows that
\[
K_{N\times N}^{\alpha,\beta}\subset (C+H^\infty)_{N\times N}\cap (C+\overline{H^\infty})_{N\times N}.
\]
Since $T(\widetilde{a})$ and $T(a)$ are invertible, from \cite[Theorem~2.94(a)]{BS06}
we deduce that $\det a$ and $\det\widetilde{a}$ belong to $\cG(C+H^\infty)$.
Hence, $\det a$ belongs to $\cG\big((C+H^\infty)\cap (C+\overline{H^\infty})\big)$.
By Corollary~\ref{co:invertibility}(c), $\det a\in \cG K^{\alpha,\beta}$.
Then, in view of \cite[Chap.~1, Theorem~1.1]{Krupnik87}, $a\in \cG K_{N\times N}^{\alpha,\beta}$.

The case $\alpha\ge 1-\beta\ge 1/2$ is treated with the help of
Corollary~\ref{co:invertibility}(a). Then $a\in \cG K_{N\times N}^{\alpha,\beta}$.
Analogously, if $\beta=\max\{\alpha,\beta\}$, then by using
Corollary~\ref{co:invertibility}(b) or (c), one cans how that
$a\in \cG K_{N\times N}^{\alpha,\beta}$.

By Simonenko's factorization theorem (see, e.g. \cite[Chap.~7, Theorem~3.2]{CG81}
or \cite[Theorem~3.14]{LS87}), if $T(a)$ is invertible on $H_N^2$, then
$a$ admits a canonical right generalized factorization in $L_N^2$, that is,
there exist functions $a_-$, $a_+$ such that $a=a_-a_+$ and $a_-^{\pm 1}\in (\overline{H^2})_{N\times N}$,
$a_+^{\pm 1}\in (H^2)_{N\times N}$. On the other hand, in view of Lemma~\ref{le:fact-property},
$a$ admits a right Wiener-Hopf factorization in $K_{N\times N}^{\alpha,\beta}$,
that is, there exist functions $u_\pm\in \cG(K_\pm^{\alpha,\beta})_{N\times N}$ such that
$a=u_-du_+$ and $d$ is a diagonal term of the form
$d=\operatorname{diag}(\chi_{\kappa_1},\dots,\chi_{\kappa_N})$.
It is clear that $u_-^{\pm 1}\in (\overline{H^2})_{N\times N}$ and
$u_+^{\pm 1}\in (H^2)_{N\times N}$. Thus $a=u_-du_+$ is a right generalized
factorization of $a$ in $L_N^2$. It is well known that the set of partial
indices of such a factorization is unique (see, e.g. \cite[Corollary~2.1]{LS87}).
Thus $d$ is the identity matrix and $a=u_-u_+$.

Since $T(\widetilde{a})$ is invertible on $H_N^2$, from \cite[Proposition~7.19(b)]{BS06}
it follows that $T(a^{-1})$ is also invertible on $H_N^2$. By what has just been
proved, there exist functions $f_\pm\in \cG(K_\pm^{\alpha,\beta})_{N\times N}$ such
that $a^{-1}=f_-f_+$. Put $v_\pm:=f_\pm^{-1}$. Then  $a=v_+v_-$ and
$v_\pm\in \cG(K_\pm^{\alpha,\beta})_{N\times N}$.
\end{proof}
\section{The Borodin-Okounkov formula}\label{sec:Borodin-Okounkov}
\subsection{Trace class operators, Hilbert-Schmidt operators, and operator determinants}
In this subsection we collect necessary facts from general operator theory
in Hilbert spaces (see \cite[Chap.~3--4]{GK69}).

Let $\cH$ and $\cK$ be separable Hilbert spaces. For a bounded linear
operator $A:\cH\to\cK$ and $n\in\Z_+$, we define
\[
s_n(A):=\inf\big\{\|A-F\|:\dim F\le n\big\}.
\]
For $1\le p<\infty$, the collection $\cC_p(\cH,\cK)$ of all bounded linear operators
$A:\cH\to\cK$ satisfying
\[
\|A\|_{\cC_p(\cH,\cK)}:=\left(\sum_{n=0}^\infty s_n^p(A)\right)^{1/p}<\infty
\]
is called the $p$-Schatten-von Neumann class. If $p=1$, then $\cC_1(\cH,\cK)$
is called the \textit{trace class} and if $p=2$, then $\cC_2(\cH,\cK)$ is called
the class of Hilbert-Schmidt operators. We will simply write $\cC_p(\cH)$ instead
of $\cC_p(\cH,\cH)$.

One can show that, for every $A\in\cC_1(\cH)$ and for every orthonormal basis
$\{\varphi_j\}_{j=0}^\infty$ of $\cH$, the series
$\sum_{j=0}^\infty\langle A\varphi_j,\varphi_j\rangle_\cH$ converges absolutely
and that its sum does not depend on the particular choice of $\{\varphi_j\}_{j=0}^\infty$.
This sum is denoted by $\operatorname{trace}A$ and is referred to as
the \textit{trace} of $A$. It is well known that
\[
|\operatorname{trace}A|\le\|A\|_{\cC_1(\cH)}
\]
for all $A\in\cC_1(\cH)$.

The Hilbert-Schmidt norm of an operator $A\in\cC_2(\cH,\cK)$ can be expressed
in the form
\[
\|A\|_{\cC_2(\cH,\cK)}=\left(\sum_{j,k}^\infty|\langle A\varphi_j,\psi_k\rangle_\cK|^2\right)^{1/2},
\]
where $\{\varphi_j\}_{j=0}^\infty$ and $\{\psi_k\}_{k=0}^\infty$ are
orthonormal bases of $\cH$ and $\cK$, respectively.

We will need the following version of the H\"older inequality.
If $B\in\cC_2(\cH,\cK)$ and $A\in\cC_2(\cK,\cH)$, then $AB\in\cC_1(\cH)$ and
\begin{equation}\label{eq:Hoelder}
\|AB\|_{\cC_1(\cH)}\le\|A\|_{\cC_2(\cK,\cH)}\|B\|_{\cC_2(\cH,\cK)}.
\end{equation}

Let $A$ be a bounded linear operator on $\cH$ of the form $I+K$ with $K\in\cC_1(\cH)$.
If $\{\lambda_j(K)\}_{j\ge 0}$ denotes the sequence of the nonzero
eigenvalues of $K$ (counted up to algebraic multiplicity), then
$\sum_{j= 0}^\infty |\lambda_j(K)|<\infty$.
Therefore the (possibly infinite) product
$\prod\limits_{j\ge 0}(1+\lambda_j(K))$
is absolutely convergent. The \textit{operator determinant} of $A$ is defined by
\[
\det A=\det (I+K)=\prod_{j\ge 0}(1+\lambda_j(K)).
\]
In the case where the spectrum of $K$ consists only of $0$ we put
$\det(I+K)=1$.
\begin{lemma}\label{le:log-det}
If $A\in\cC_1(\cH)$ and $\|A\|_{\cC_1(\cH)}<1$, then
$|\log\det(I-A)|\le 2\|A\|_{\cC_1(\cH)}$.
\end{lemma}
\begin{proof}
Since $A\in\cC_1(\cH)$, by formula (1.16) of \cite[Chap.~IV]{GK69},
\begin{equation}\label{eq:log-det-1}
\log\det(I-A)=\operatorname{trace}\log(I-A).
\end{equation}
On the other hand,
\begin{equation}\label{eq:log-det-2}
\log(I-A)=-\sum_{j=1}^\infty\frac{1}{j}A^j.
\end{equation}
From \eqref{eq:log-det-1}, \eqref{eq:log-det-2}, and $|\operatorname{trace}A|\le\|A\|_{\cC_1(\cH)}$
we get
\begin{equation}\label{eq:log-det-3}
|\log\det(I-A)|
\le
\left|\operatorname{trace}\left[\sum_{j=1}^\infty\frac{1}{j}A^j\right]\right|
\le
\sum_{j=1}^\infty|\operatorname{trace}A^j|
\le
\sum_{j=1}^\infty\|A^j\|_{\cC_1(\cH)}.
\end{equation}
By H\"older's inequality,
\[
\|A^j\|_{\cC_1(\cH)}\le\|A\|_{\cC_1(\cH)}\|A^{j-1}\|_{\cC_\infty(\cH)}
\le
\|A\|_{\cC_1(\cH)}\|A\|_{\cC_\infty(\cH)}^{j-1}.
\]
Taking into account that $\|A\|_{\cC_\infty(\cH)}\le\|A\|_{\cC_1(\cH)}$, we get
$\|A^j\|_{\cC_1(\cH)}\le\|A\|_{\cC_1(\cH)}^j$. Hence, \eqref{eq:log-det-3}
yields
\[
|\log\det(I-A)|\le\sum_{j=1}^\infty\|A\|_{\cC_1(\cH)}^j=\frac{\|A\|_{\cC_1(\cH)}}{1-\|A\|_{\cC_1(\cH)}}
\le 2\|A\|_{\cC_1(\cH)}
\]
because $\|A\|_{\cC_1(\cH)}<1$.
\end{proof}
\subsection{The Borodin-Okounkov formula under weakened hypotheses}
For $a\in L_{N\times N}^\infty$ and $n\in\Z_+$, define the operators
\[
P_n:\sum_{k=0}^\infty\widehat{a}(k)\chi_k\mapsto\sum_{k=0}^n\widehat{a}(k)\chi_k,
\quad Q_n:=I-P_n.
\]
The operator $P_nT(a)P_n:P_nH_N^2\to P_n H_N^2$ may be identified with the
finite block Toeplitz matrix $T_n(a)=(\widehat{a}(j-k))_{j,k=0}^n$.

In June 1999, Its and Deift raised the question whether there is a general
formula that expresses the determinant of the Toeplitz matrix $T_n(a)$ as the operator
determinant of an operator $I-K$ where $K$ acts on $\ell_2\{n+1,n+2,\dots\}$.
Borodin and Okunkov showed in 2000 that such a formula exists (however, it was
known even much earlier. In 1979, Geronimo and Case used it to prove the
strong Szeg\H{o} limit theorem). Further, in 2000, several different proofs
of it were found by Basor and Widom and by B\"ottcher. We refer to the books
by Simon \cite{Simon05}, B\"ottcher and Grudsky \cite[Section~2.8]{BG05},
B\"ottcher and Silbermann \cite[Section~10.40]{BS06} for the exact references,
proofs, and historical remarks on this beautiful piece of mathematics. Below
we formulate the Borodin-Okounkov formula in a form suggested by Widom under
assumptions on the symbol $a$ of $T_n(a)$ which are slightly weaker than in
\cite[Section~10.40]{BS06}.
\begin{theorem}\label{th:BO}
Suppose $a\in (C+H^\infty)_{N\times N}\cup (C+\overline{H^\infty})_{N\times N}$
satisfies the following hypothesis:
\begin{enumerate}
\item[(i)]
there are two factorizations $a=u_-u_+=v_+v_-$, where
$u_-,v_-\in \cG(\overline{H^\infty})_{N\times N}$ and
$u_+,v_+\in \cG(H^\infty)_{N\times N}$;
\end{enumerate}
Put $b:=v_-u_+^{-1}$ and $c:=u_-^{-1}v_+$. Suppose that
\begin{enumerate}
\item[(ii)]
$H(b)H(\widetilde{c})\in\cC_1(H_N^2)$.
\end{enumerate}
Then the constants
\begin{equation}\label{eq:def-G}
G(a):=\lim_{r\to 1-0}\exp\left(\frac{1}{2\pi}\int_0^{2\pi}
\log\det a_r(e^{i\theta})d\theta\right),
\end{equation}
where
\[
a_r(e^{i\theta}):=\sum_{n=-\infty}^{\infty}
\widehat{a}(n)r^{|n|}e^{in\theta},
\]
and
\[
E(a):=\frac{1}{\det T(b)T(c)}
\]
are well defined, are not equal to zero, and the Borodin-Okounkov formula
\begin{equation}\label{eq:BO}
\det T_n(a)=G(a)^{n+1}E(a)\det\big(I-Q_nH(b)H(\widetilde{c})Q_n\big)
\end{equation}
holds for every $n\in\N$. If, in addition,
\begin{enumerate}
\item[(iii)]
$H(a)H(\widetilde{a}^{-1})\in\cC_1(H_N^2)$,
\end{enumerate}
then $E(a)=\det T(a)T(a^{-1})$.
\end{theorem}
\begin{proof}
From (i) and \cite[Proposition~2.14]{BS06} it follows that the operators
$T(a)$ and $T(a^{-1})$ are invertible on $H_N^2$. If $a\in(C+H^\infty)_{N\times N}$,
then from \cite[Proposition~10.6(a)]{BS06} we deduce that the limit in \eqref{eq:def-G}
exists, is finite and nonzero. Hence the constant $G(a)$ is well defined.
The proof of \cite[Proposition~10.6(a)]{BS06} equally works also for
the case $a\in(C+\overline{H^\infty})_{N\times N}$. Therefore, $G(a)$
is well defined in this case as well.

From \cite[Proposition~2.14]{BS06} it follows also that the operators $T(b)$,
$T(c)$ are invertible and $T(b)T(c)=I-H(b)H(\widetilde{c})$. From (ii)
and the above equality we get that $\det T(b)T(c)$ makes sense. Since
$T(b)T(c)$ is invertible, $\det T(b)T(c)\ne 0$ and therefore the constant $E(a)$
is well defined.

The Borodin-Okounkov formula \eqref{eq:BO} is proved in \cite[Section~10.40]{BS06}
under the assumption that $a\in K_{N\times N}^{1/2,1/2}$ admits right and
left canonical Wiener-Hopf factorizations in $K_{N\times N}^{1/2,1/2}$.
The two proofs given in \cite[Section~10.40]{BS06} work equally under
weaker hypotheses (i)--(ii).

Applying \cite[Proposition~2.14]{BS06}, we get
\[
I-H(b)H(\widetilde{c})=T(b)T(c)=T(v_-)T(u_+^{-1})T(u_-^{-1})T(v_+)
\]
and
\[
\begin{split}
T(v_+)\big(I-H(b)H(\widetilde{c})\big)T^{-1}(v_+)
&=
I-T(v_+)H(b)H(\widetilde{c})T^{-1}(v_+)
\\
&=
T(v_+)T(v_-)T(u_+^{-1})T(u_-^{-1})
\\
&=
T^{-1}(a^{-1})T^{-1}(a).
\end{split}
\]
From these equalities and \cite[Chap.~IV, Section~1.6]{GK69}
it follows that
\begin{equation}\label{eq:BO-proof-1}
\begin{split}
\det T(b)T(c)
&=\det\big(I-H(b)H(\widetilde{c})\big)
\\
&=\det\big(I-T(v_+)H(b)H(\widetilde{c})T^{-1}(v_+)\big)
\\
&=\det T^{-1}(a^{-1})T^{-1}(a).
\end{split}
\end{equation}
From (iii) and $T(a)T(a^{-1})=I-H(a)H(\widetilde{a}^{-1})$ it follows that
$\det T(a)T(a^{-1})$ makes sense. By \cite[Chap.~IV, Section~1.7]{GK69},
\[
\det T^{-1}(a^{-1})T^{-1}(a)\cdot\det T(a)T(a^{-1})=
\det T^{-1}(a^{-1})T^{-1}(a)T(a)T(a^{-1})=1.
\]
Hence
\begin{equation}\label{eq:BO-proof-2}
\det T^{-1}(a^{-1})T^{-1}(a)=\frac{1}{\det T(a)T(a^{-1})}.
\end{equation}
Combining \eqref{eq:BO-proof-1} and \eqref{eq:BO-proof-2}, we arrive at $E(a)=\det T(a)T(a^{-1})$.
\end{proof}
\section{Proof of the main result}\label{sec:proof}
\subsection{Hilbert-Schmidt norms of truncations of Hankel operators}
Let $\gamma\in\R$. By $\ell_2^\gamma$ we denote the Hilbert space of all
sequences $\{\varphi_k\}_{k=0}^\infty$ such that
\[
\sum_{k=0}^\infty |\varphi_k|^2(k+1)^{2\gamma}<\infty.
\]
Clearly, the sequence $\{e_k/(k+1)^\gamma\}_{k=0}^\infty$, where
$(e_k)_j=\delta_{kj}$ and $\delta_{kj}$ is the Kronecker delta, is an orthonormal basis
of $\ell_2^\gamma$. If $\gamma=0$, we will simply write $\ell_2$ instead of
$\ell_2^0$.

In this subsection we will estimate Hilbert-Schmidt norms of truncations of Hankel
operators acting between $\ell_2$ and $\ell_2^\gamma$ by the rules
\[
\begin{split}
&
H(a):\{\varphi_j\}_{j=0}^\infty\mapsto\Big\{\sum_{j=0}^\infty \widehat{a}(i+j+1)\varphi_j\Big\}_{i=0}^\infty,
\\
&
H(\widetilde{a}):\{\varphi_j\}_{j=0}^\infty\mapsto\Big\{\sum_{j=0}^\infty \widehat{a}(-i-j-1)\varphi_j\Big\}_{i=0}^\infty.
\end{split}
\]
Notice that one can identify Hankel operators acting on $H^2$ and on $\ell^2$.
For $\varphi=\{\varphi_j\}_{j=0}^\infty$ and $n\in\Z_+$, define
\[
(Q_n\varphi)_j=\left\{\begin{array}{lc}
\varphi_j &\text{if }j\ge n+1,\\
0 &\text{otherwise}.
\end{array}\right.
\]
For $a\in K^{\alpha,\beta}$ and $n\in\N$, put
\[
\begin{split}
r_n^-(a) &:=\left(\sum_{k=n+1}^\infty |\widehat{a}(-k)|^2(k+1)^{2\alpha}\right)^{1/2},
\\
r_n^+(a) &:=\left(\sum_{k=n+1}^\infty |\widehat{a}(k)|^2(k+1)^{2\beta}\right)^{1/2}.
\end{split}
\]
\begin{lemma}\label{le:HS}
Let $-1/2<\gamma<1/2$ and $\alpha,\beta\in(0,1)$. Suppose $b,c\in K^{\alpha,\beta}$.
\begin{enumerate}
\item[(a)]
If $\alpha\ge\gamma+1/2$, then there exists a positive constant
$M(\alpha,\gamma)$ depending only on $\alpha$ and $\gamma$ such that
for all sufficiently large $n$,
\begin{equation}\label{eq:HS-1}
\|H(\widetilde{c})Q_n\|_{\cC_2(\ell_2,\ell_2^\gamma)}
\le
M(\alpha,\gamma)
\frac{r_{n+1}^-(c)}{n^{\alpha-\gamma-1/2}}.
\end{equation}

\item[(b)]
If $\beta\ge-\gamma+1/2$, then there exists a positive constant $M(\beta,\gamma)$
depending only on $\beta$ and $\gamma$ such that for all sufficiently large $n$,
\begin{equation}\label{eq:HS-2}
\|Q_nH(b)\|_{\cC_2(\ell_2^\gamma,\ell_2)}
\le
M(\beta,\gamma)
\frac{r_{n+1}^+(b)}{n^{\beta+\gamma-1/2}}.
\end{equation}
\end{enumerate}
\end{lemma}
\begin{proof}
(a) It is easy to see that
\[
(H(\widetilde{c})Q_n e_j)_k=\left\{\begin{array}{cc}
\widehat{c}(-k-j-1) & \text{if}\ j\ge n+1, \ k\in\Z_+,
\\[3mm]
0 & \text{otherwise}.
\end{array}\right.
\]
Then
\begin{equation}\label{eq:HS-3}
\begin{split}
\|H(\widetilde{c})Q_n\|_{\cC_2(\ell_2,\ell_2^\gamma)}^2
&=
\sum_{j,k=0}^\infty
\left|\left\langle H(\widetilde{c})Q_ne_j,\frac{e_k}{(k+1)^\gamma}\right\rangle_{\ell_2^\gamma}\right|^2
\\
&=\sum_{j,k=0}^\infty|(H(\widetilde{c})Q_ne_j)_k|^2(k+1)^{2\gamma}
\\
&=\sum_{k=0}^\infty\sum_{j=n+1}^\infty |\widehat{c}(-k-j-1)|^2(k+1)^{2\gamma}
\\
&=\sum_{k=n+2}^\infty|\widehat{c}(-k)|^2\sum_{j=1}^{k-n-1}j^{2\gamma}.
\end{split}
\end{equation}
If $-1/2<\gamma<1/2$, then
\begin{equation}\label{eq:HS-4}
\sum_{j=1}^{k-n-1}j^{2\gamma}\le\frac{(k-n)^{1+2\gamma}}{1+2\gamma}.
\end{equation}
From \eqref{eq:HS-3} and \eqref{eq:HS-4} it follows that
\begin{equation}\label{eq:HS-5}
\begin{split}
\|H(\widetilde{c})Q_n\|_{\cC_2(\ell_2,\ell_2^\gamma)}^2
&\le
\frac{1}{1+2\gamma}\sum_{k=n+2}^\infty|\widehat{c}(-k)|^2(k-n)^{1+2\gamma}
\\
&\le
\frac{1}{1+2\gamma}\sum_{k=n+2}^\infty|\widehat{c}(-k)|^2(k+1)^{2\alpha}
\frac{(k-n)^{1+2\gamma}}{k^{2\alpha}}
\\
&\le
\frac{1}{1+2\gamma}
\left(\sup_{k\ge n+2}\frac{(k-n)^{1+2\gamma}}{k^{2\alpha}}\right)
\big[r_{n+1}^-(c)\big]^2.
\end{split}
\end{equation}

If $1+2\gamma=2\alpha$, then
\begin{equation}\label{eq:HS-6}
\sup_{k\ge n+2}\left(\frac{k-n}{k}\right)^{2\alpha}\le 1.
\end{equation}
Combining \eqref{eq:HS-5} and \eqref{eq:HS-6}, we arrive at \eqref{eq:HS-1}
with $M(\alpha,\gamma)=(2\alpha)^{-1/2}$.

If $\alpha>\gamma+1/2$, then put
\[
A:=\frac{2\gamma+1}{2\alpha-2\gamma-1}>0.
\]
Let $n\ge 2/A$. Then
\[
x_n:=(A+1)n=\frac{2\alpha n}{2\alpha-2\gamma-1}\in[n+2,\infty).
\]
It is not difficult to show that the function
\[
f_n(x):=(x-n)^{1+2\gamma}x^{-2\alpha},
\quad
x\in[n+2,\infty)
\]
attains its absolute maximum at $x_n$. Thus
\begin{equation}\label{eq:HS-7}
\sup_{k\ge n+2}\frac{(k-n)^{1+2\gamma}}{k^{2\alpha}}
\le
f_n(x_n)=\frac{A^{1+2\gamma}}{(A+1)^{2\alpha}}n^{1+2\gamma-2\alpha}.
\end{equation}
Combining \eqref{eq:HS-5} and \eqref{eq:HS-7}, we arrive at
\eqref{eq:HS-1} with
\[
M(\alpha,\gamma):=(1+2\gamma)^{-1/2}A^{1/2+\gamma}(A+1)^{-\alpha}
\]
for all $n\ge 2/A$. Part (a) is proved.
The proof of part (b) is analogous.
\end{proof}
\subsection{Trace class norms of truncations of products of two Hankel operators}
The following fact is well known (see e.g. \cite[Section~10.12]{BS06}
and also \cite{Silbermann81}, \cite[Chap.~6]{Peller03}).
\begin{lemma}\label{le:product-trace}
Let $N\ge 1$, $\alpha,\beta\in(0,1)$, and $\alpha+\beta\ge 1$. If
$b,c\in K_{N\times N}^{\alpha,\beta}$, then $H(b)H(\widetilde{c})\in\cC_1(H_N^2)$.
\end{lemma}
We will also need a quantitative version of the above result for truncations of the
product $H(b)H(\widetilde{c})$.

For $a\in K_{N\times N}^{\alpha,\beta}$ and $n\in\N$, put
\[
\begin{split}
R_n^-(a) &:=\left(\sum_{k=n+1}^\infty \|\widehat{a}(-k)\|_{\C_{N\times N}}^2(k+1)^{2\alpha}\right)^{1/2},
\\
R_n^+(a) &:=\left(\sum_{k=n+1}^\infty \|\widehat{a}(k)\|_{\C_{N\times N}}^2(k+1)^{2\beta}\right)^{1/2}.
\end{split}
\]
\begin{lemma}\label{le:TC}
Let $N\ge 1$, $\alpha,\beta\in(0,1)$, and $\alpha+\beta\ge 1$. Then there exists
a constant $L=L_{\alpha,\beta,N}$ depending only on $N$ and $\alpha,\beta$
such that for every $b,c\in K_{N\times N}^{\alpha,\beta}$ and all sufficiently large $n$,
\begin{equation}\label{eq:TC-1}
\|Q_nH(b)H(\widetilde{c})Q_n\|_{\cC_1(H_N^2)}
\le
\frac{L}{n^{\alpha+\beta-1}}R_{n+1}^+(b)R_{n+1}^-(c).
\end{equation}
\end{lemma}
\begin{proof}
Put $\gamma:=1/2-\beta$. Then $\gamma\in(-1/2,1/2)$ and $\alpha\ge\gamma+1/2$
because $\alpha+\beta\ge 1$. Let $b_{ij}$ and $c_{ij}$, where $i,j\in\{1,\dots,N\}$,
be the entries of $b,c\in K_{N\times N}^{\alpha,\beta}$, respectively.
By Lemma~\ref{le:HS}, there exist positive constants $M(\alpha,\gamma)$ and
$M(\beta,\gamma)$ depending only on $\alpha,\beta$, and $\gamma$ such that for
all sufficiently large $n$ and all $i,j\in\{1,\dots,N\}$,
\begin{eqnarray}
\|Q_nH(b_{ij})\|_{\cC_2(\ell_2^\gamma,\ell_2)}
&\le&
M(\beta,\gamma)\frac{r_{n+1}^+(b_{ij})}{n^{\beta+\gamma-1/2}},
\label{eq:TC-2}
\\
\|H(\widetilde{c_{ij}})Q_n\|_{\cC_2(\ell_2,\ell_2^\gamma)}
&\le&
M(\alpha,\gamma)\frac{r_{n+1}^-(c_{ij})}{n^{\alpha-\gamma-1/2}}.
\label{eq:TC-3}
\end{eqnarray}
From H\"older's inequality \eqref{eq:Hoelder} and \eqref{eq:TC-2}--\eqref{eq:TC-3} it follows that
\begin{equation}\label{eq:TC-4}
\begin{split}
\|Q_nH(b_{ij})H(\widetilde{c_{ij}})Q_n\|_{\cC_1(\ell_2)}
&\le
\|Q_nH(b_{ij})\|_{\cC_2(\ell_2^\gamma,\ell_2)}
\|H(\widetilde{c_{ij}})Q_n\|_{\cC_2(\ell_2,\ell_2^\gamma)}
\\
&\le
M(\alpha,\gamma)M(\beta,\gamma)\frac{r_{n+1}^+(b_{ij})r_{n+1}^-(c_{ij})}{n^{\alpha+\beta-1}}
\end{split}
\end{equation}
for all $i,j\in\{1,\dots,N\}$ and all large $n$.

It is not difficult to verify
that there exist positive constants $A_N$ and $B_N$ depending only
on the dimension $N$ such that
\begin{equation}\label{eq:TC-5}
\|Q_n H(b)H(\widetilde{c})Q_n\|_{\cC_1(H_N^2)}
\le
A_N\max_{1\le i,j\le N}
\|Q_n H(b_{ij})H(\widetilde{c_{ij}})Q_n\|_{\cC_1(\ell_2)}
\end{equation}
and
\begin{equation}\label{eq:TC-6}
\max_{1\le i,j\le N}r_{n+1}^+(b_{ij})
\le
B_N R_{n+1}^+(b),
\quad
\max_{1\le i,j\le N}r_{n+1}^-(c_{ij})
\le
B_N R_{n+1}^-(c)
\end{equation}
for all sufficiently large $n$ and all $b,c\in K_{N\times N}^{\alpha,\beta}$.
Combining \eqref{eq:TC-4}--\eqref{eq:TC-6}, we arrive at \eqref{eq:TC-1}
with $L=L_{\alpha,\beta,N}:=A_NB_N^2M(\alpha,\gamma)M(\beta,\gamma)$.
\end{proof}
\subsection{Tails of the norms of functions in generalized Krein algebras}
\begin{lemma}\label{le:uniform}
Let $N\ge 1$, $\alpha,\beta\in(0,1)$, and $\max\{\alpha,\beta\}\ge 1/2$.
Suppose $\Sigma$ is a compact set in the complex plane.
If $a:\Sigma\to K_{N\times N}^{\alpha,\beta}$ is a continuous function, then
\begin{equation}\label{eq:uniform-1}
\lim_{n\to\infty}\sup_{\lambda\in\Sigma} R_n^-\big(a(\lambda)\big)=0,
\quad
\lim_{n\to\infty}\sup_{\lambda\in\Sigma} R_n^+\big(a(\lambda)\big)=0.
\end{equation}
\end{lemma}
\begin{proof}
This statement is proved by analogy with \cite[Proposition~2.3]{K-ZAA}
and \cite[Lemma~6.2]{K-WOAT06}. Let us prove the first equality in
\eqref{eq:uniform-1}. Assume the contrary. Then there exist a constant $C>0$
and a sequence $\{n_k\}_{k=1}^\infty$ such that
\[
\lim_{k\to\infty}\sup_{\lambda\in\Sigma}R_{n_k}^-\big(a(\lambda)\big)\ge C.
\]
Hence there are a number $k_0\in\N$ and a sequence $\{\lambda_k\}_{k=k_0}^\infty$
such that for all $k\ge k_0$,
\begin{equation}\label{eq:uniform-2}
R_{n_k}^-\big(a(\lambda_k)\big)\ge \frac{C}{2}>0.
\end{equation}
Since $\{\lambda_k\}_{k=k_0}^\infty$ is bounded, there is its convergent
subsequence $\{\lambda_{k_j}\}_{j=1}^\infty$. Let $\lambda_0$ be the limit
of this subsequence. Clearly, $\lambda_0\in\Sigma$ because $\Sigma$ is closed.
Since the function $a:\Sigma\to K_{N\times N}^{\alpha,\beta}$ is continuous
at $\lambda_0$, for every $\eps\in(0,C/2)$, there exists a $\Delta>0$ such that
$|\lambda-\lambda_0|<\Delta,\lambda\in\Sigma$ implies
$\|a(\lambda)-a(\lambda_0)\|_{K_{N\times N}^{\alpha,\beta}}<\eps$.
Because $\lambda_{k_j}\to\lambda_0$ as $j\to\infty$, for that $\Delta$
there is a number $J\in\N$ such that $|\lambda_{k_j}-\lambda_0|<\Delta$
for all $j\ge J$, and thus
\begin{equation}\label{eq:uniform-3}
\|a(\lambda_{k_j})-a(\lambda_0)\|_{K_{N\times N}^{\alpha,\beta}}<\eps
\quad\text{for all}\quad j\ge J.
\end{equation}
On the other hand, \eqref{eq:uniform-2} implies that
\begin{equation}\label{eq:uniform-4}
R_{n_{k_j}}^-\big(a(\lambda_{k_j})\big)\ge \frac{C}{2}>0
\quad\text{for all}\quad j\ge J.
\end{equation}
By the Minkowski inequality,
\begin{equation}\label{eq:uniform-5}
\begin{split}
R_{n_{k_j}}^-\big(a(\lambda_{k_j})\big)
&\le
R_{n_{k_j}}^-\big(a(\lambda_0)\big)+R_{n_{k_j}}^-\big(a(\lambda_{k_j})-a(\lambda_0)\big)
\\
&\le
R_{n_{k_j}}^-\big(a(\lambda_0)\big)+
\|a(\lambda_{k_j})-a(\lambda_0)\|_{K_{N\times N}^{\alpha,\beta}}.
\end{split}
\end{equation}
From \eqref{eq:uniform-3}--\eqref{eq:uniform-5} we get for all $j\ge J$,
\[
R_{n_{k_j}}^-\big(a(\lambda_0)\big)\ge\frac{C}{2}-\eps>0.
\]
Therefore,
\[
\sum_{k=0}^\infty\big\|[a(\lambda_0)]\widehat{\hspace{2mm}}(-k)\big\|_{\C_{N\times N}}^2(k+1)^{2\alpha}=+\infty
\]
and this contradicts the fact that $a(\lambda_0)\in K_{N\times N}^{\alpha,\beta}$.
Hence, the first equality in \eqref{eq:uniform-1} is proved. The second equality
in \eqref{eq:uniform-1} can be proved by analogy.
\end{proof}
\subsection{Proof of Theorem~\ref{th:main}}
\begin{proof}
The proof is developed similarly to the proofs of \cite[Theorem~1.5]{K-ZAA},
\cite[Theorem~1.4]{K-IWOTA05}, \cite[Theorem~2.2]{K-WOAT06} with some
modifications. For the convenience of the reader, we provide some details.

Without loss of generality, we can suppose that $\max\{\alpha,\beta\}>1/2$
(otherwise $\max\{\alpha,\beta\}\le 1/2$ and $\alpha+\beta\ge 1$ imply that
$\alpha=\beta=1/2$ and this is exactly the case of Theorem~\ref{th:Widom}).

Let $\Omega$ be a bounded open set containing the set
$\operatorname{sp}T(a)\cup\operatorname{sp}T(\widetilde{a})$ on the closure
of which $f$ is analytic. From \eqref{eq:Krein-embedding-2} and
\cite[Theorem~7.20]{BS06} it follows that $\Omega$ contains the spectrum
(eigenvalues) of $T_n(a)$ for all sufficiently large $n$. Further,
Corollary~\ref{co:invertibility} and Theorem~\cite[Theorem~2.94]{BS06}
imply that the spectrum of $a$ in $K_{N\times N}^{\alpha,\beta}$
is contained in $\Omega$. Hence $f(a)\in K_{N\times N}^{\alpha,\beta}$
and $f(T_n(a))$ is well defined whenever $f$ is analytic on
$\operatorname{sp}T(a)\cup\operatorname{sp}T(\widetilde{a})$.

One can choose a closed neighborhood $\Sigma$
of its boundary $\partial\Omega$ such that $f$ is analytic on $\Sigma$ and
$\Sigma\cap(\operatorname{sp}T(a)\cup\operatorname{sp}T(\widetilde{a}))=\emptyset$.
If $\lambda\in\Sigma$, then $T(a)-\lambda I=T[a-\lambda]$ and
$T(\widetilde{a})-\lambda I=T[(a-\lambda)\widetilde{\hspace{2mm}}]$ are invertible
on $H_N^2$. From Theorem~\ref{th:Krein-WH} it follows that
$(a-\lambda)^{-1}\in K_{N\times N}^{\alpha,\beta}$ and that $a-\lambda$ admits
canonical right and left Wiener-Hopf factorizations
\[
a-\lambda=u_-(\lambda)u_+(\lambda)=v_+(\lambda)v_-(\lambda)
\]
in the algebra $K_{N\times N}^{\alpha,\beta}$. Since
$a-\lambda:\Sigma\to K_{N\times N}^{\alpha,\beta}$ is a continuous function with
respect to $\lambda$, from Lemma~\ref{le:fact-property} and
Theorem~\ref{th:stability} we that these canonical factorizations
can be chosen so that the functions
\[
u_-^{\pm 1},v_-^{\pm 1}:\Sigma\to(K^{\alpha,\beta}\cap \overline{H^\infty})_{N\times N},
\quad
u_+^{\pm 1},v_+^{\pm 1}:\Sigma\to(K^{\alpha,\beta}\cap H^\infty)_{N\times N}
\]
are continuous with respect to $\lambda\in\Sigma$.

Put $b(\lambda):=v_-(\lambda)u_+^{-1}(\lambda)$ and $c(\lambda):=u_-^{-1}(\lambda)v_+(\lambda)$.
By Lemma~\ref{le:product-trace},
\begin{equation}\label{eq:main-1}
H[b(\lambda)]H[\widetilde{c(\lambda)}]\in\cC_1(H_N^2)
\end{equation}
for all $\lambda\in\Sigma$. On the other hand, since $(a-\lambda)^{-1}\in K_{N\times N}^{\alpha,\beta}$,
we also have
$(\widetilde{a}-\lambda)^{-1}=[(a-\lambda)\widetilde{\hspace{2mm}}]^{-1}\in K_{N\times N}^{\alpha,\beta}$.
Then from Lemma~\ref{le:product-trace} it follows that
\begin{equation}\label{eq:main-2}
H[a-\lambda]H[(\widetilde{a}-\lambda)^{-1}]\in\cC_1(H_N^2)
\end{equation}
for all $\lambda\in\Sigma$. From \eqref{eq:Krein-embedding-2},
\eqref{eq:main-1}--\eqref{eq:main-2}, and Theorem~\ref{th:BO}
we conclude that
\begin{equation}\label{eq:main-3}
\begin{split}
\det T_n(a-\lambda)
= &G(a-\lambda)^{n+1}\det T[a-\lambda]T[(a-\lambda)^{-1}]
\\
&\times
\det\big(I-Q_nH[b(\lambda)]H[\widetilde{c(\lambda)}]Q_n\big)
\end{split}
\end{equation}
for all $\lambda\in\Sigma$ and all $n\in\N$.

From Lemmas~\ref{le:TC} and~\ref{le:uniform} it follows that there exists a number
$n_0\in\N$ such that
\begin{eqnarray}\label{eq:main-4}
&&
\big\|Q_nH[b(\lambda)]H[\widetilde{c(\lambda)}]Q_n\big\|_{\cC_1(H_N^2)}
\le\frac{L
\left(\sup\limits_{\lambda\in\Sigma}R_n^+\big(b(\lambda)\big)\right)
\left(\sup\limits_{\lambda\in\Sigma}R_n^-\big(c(\lambda)\big)\right)
}{n^{\alpha+\beta-1}}<1
\end{eqnarray}
for all $\lambda\in\Sigma$ and all $n\ge n_0$. Here $L$ is a positive
constant depending only only on $\alpha,\beta$ an $N$. Combining
\eqref{eq:main-4} and Lemma~\ref{le:log-det}, we arrive at the estimate
\begin{equation}\label{eq:main-5}
\begin{split}
&
\big|\log\det\big(I-Q_nH[b(\lambda)]H[\widetilde{c(\lambda)}]Q_n\big)\big|
\\
&\le\frac{2L}{n^{\alpha+\beta-1}}
\left(\sup\limits_{\lambda\in\Sigma}R_n^+\big(b(\lambda)\big)\right)
\left(\sup\limits_{\lambda\in\Sigma}R_n^-\big(c(\lambda)\big)\right)
\end{split}
\end{equation}
for all $\lambda\in\Sigma$ and all $n\ge n_0$.

From \eqref{eq:main-3},
\eqref{eq:main-5}, and Lemma~\ref{le:uniform} we conclude that
\begin{equation}\label{eq:main-6}
\begin{split}
\log\det T_n(a-\lambda)
=&
(n+1)\log G(a-\lambda)
\\
&+\log\det T[a-\lambda]T[(a-\lambda)^{-1}]+o(n^{1-\alpha-\beta})
\end{split}
\end{equation}
as $n\to\infty$ uniformly with respect to $\lambda\in\Sigma$. Hence we can differentiate
both sides with respect to $\lambda$, multiply by $f(\lambda)$, and
integrate over $\partial\Omega$. The rest of the proof is the repetition
of Widom's arguments \cite[Theorem~6.2]{Widom76} (see also
\cite[Section~10.90]{BS06} and \cite[Theorem~5.6]{BS99}) with
$o(1)$ replaced by $o(n^{1-\alpha-\beta})$.
\end{proof}
\begin{remark}
Formula \eqref{eq:main-6} for $\lambda=0$ and $a\in K_{N\times N}^{\alpha,\alpha}$
with $\alpha>1/2$ was obtained by Silbermann \cite{Silbermann81} by using methods
of \cite{BS80} (see also \cite[Sections~6.15--6.23]{BS83} and
\cite[Sections~10.34--10.37]{BS06}). On the other hand, for $\alpha+\beta=1$,
formula \eqref{eq:main-6} with $\lambda=0$ was proved by B\"ottcher and
Silbermann \cite[Theorem~6.11]{BS83} and \cite[Theorem~10.30]{BS06}.
\end{remark}

\end{document}